\documentclass{scrartcl}
\usepackage[english]{babel}

\usepackage{amsmath,amssymb,amsthm,amsfonts}
\usepackage{exscale}
\usepackage{mathtools}

\usepackage[T1]{fontenc}
\usepackage{palatino}
\usepackage{microtype}

\usepackage{bbm,mathrsfs} 

\usepackage{graphicx,xcolor}
\usepackage{textcomp}

\usepackage{shuffle}

\usepackage[numbers,sort&compress]{natbib}
\definecolor{links}{rgb}{0,0.3,0}
\usepackage[colorlinks=true,citecolor=blue,urlcolor=links,breaklinks=true]{hyperref}

\newcommand{\Filename}[1]{{\upshape\ttfamily #1}}

\numberwithin{equation}{section}

\newcommand{\Q}{
\mathbbm{Q}%
}
\newcommand{\N}{
\mathbbm{N}%
}
\newcommand{\Z}{
\mathbbm{Z}%
}
\newcommand{\C}{
\mathbbm{C}%
}

\newcommand{\tp}{\otimes}

\newcommand{\conjugate}[1]{#1^{\ast}}

\newcommand{\dd}[1][]{\mathrm{d}^{#1}}
\newcommand{\restrict}[2]{%
{\left. #1 \right|}_{#2}%
}
\DeclareMathOperator{\lin}{lin}

\newcommand{\defas}{
\mathrel{\mathop:}=
}
\newcommand{\set}[1]{
\left\{ #1 \right\}
}
\newcommand{\setexp}[2]{
\left\{ #1\!:\ #2 \right\}
}
\newcommand{\abs}[1]{
\left\lvert #1 \right\rvert
}

\newcommand{\ber}[2]{\mathcal{B}_{#1}(#2)}
\newcommand{\vect}[1]{\boldsymbol{#1}}
\newcommand{\PLi}{\operatorname{PLi}}
\newcommand{\MPLs}{\mathscr{L}}
\newcommand{\Inversion}{\Phi}
\newcommand{\letter}[1]{\omega_{#1}}
\DeclareMathOperator{\Li}{Li}
\newcommand{\mzv}[2][]{\zeta^{#1}({#2})}
\newcommand{\MZV}{\mathcal{Z}}
\newcommand{\RU}[1]{\mu_{#1}}

\newcommand{\MapleNote}{%
	\footnote{Maple is a trademark of Waterloo Maple Inc.}
}
\newcommand{\MapleTM}{%
	\href{http://www.maplesoft.com/products/Maple/}
	{\textsf{\textup{Maple}}\texttrademark}%
}

\newcommand{\AnaReg}[2]{\Reglim_{#1 \rightarrow #2}}
\DeclareMathOperator*{\Reglim}{Reg}

\newcommand{\imag}{i}
\DeclareMathOperator{\Imaginaerteil}{Im}
\DeclareMathOperator{\Realteil}{Re}

\newtheorem{theorem}{Theorem}[section]
\newtheorem{lemma}[theorem]{Lemma}
\newtheorem{corollary}[theorem]{Corollary}
\newtheorem{conjecture}[theorem]{Conjecture}
\newtheorem{example}[theorem]{Example}
\theoremstyle{definition}
\newtheorem{definition}[theorem]{Definition}
\newtheorem{remark}[theorem]{Remark}

\newcommand{\mytitle}{The parity theorem for multiple polylogarithms}

\hypersetup{%
	pdftitle = {\mytitle},
	pdfauthor = {Erik Panzer}
}

\begin{document}

\title{\vspace{-1cm}\mytitle}

\author{\href{mailto:erik.panzer@all-souls.ox.ac.uk}{Erik Panzer}\footnote{%
All Souls College, OX1 4AL, Oxford, UK%
}}
\maketitle
\abstract{%
	We generalize the well-known parity theorem for multiple zeta values (MZV) to functional equations of multiple polylogarithms (MPL).
	This reproves the parity theorem for MZV with an additional integrality statement, and also provides parity theorems for special values of MPL at roots of unity (also known as coloured MZV).
	We give explicit formulas in depths 2 and 3 and provide a computer program to compute the functional equations.%
}

\section{Introduction}
\label{sec:intro}

Multiple zeta values (MZV) are defined for integers $\vect{n}\in\N^d$ with $n_d>1$ as
\begin{equation}
	\mzv{\vect{n}}
	=
	\mzv{n_1,\ldots,n_d}
	\defas 
	\sum_{0<k_1<\cdots<k_d} \frac{1}{k_1^{n_1}\!\cdots k_d^{n_d}}
	\label{eq:def-MZV}%
\end{equation}
where $d$ is called \emph{depth} and $\abs{\vect{n}}=n_1+\cdots+ n_d$ is the \emph{weight} \cite{Hoffman:MultipleHarmonicSeries,Zagier:MZVApplications}. 
We set $\N^0 \defas \set{\emptyset}$ and $\mzv{\emptyset} \defas 1$ in weight $\abs{\emptyset} \defas 0$ and we write
\begin{equation}\!
	\MZV^{d}_w
	\defas \lin_{\Q} \setexp{\mzv[k]{2} \mzv{\vect{n}} }{
		\vect{n}\in\N^r,
		k\in\N_0,
		n_r>1,
		\abs{\vect{n}}+2k=w,
		r \leq d
	}
	\label{eq:MZV-depth-filtration}%
\end{equation}
for all rational linear combinations of MZV with weight $w$ and depth at most $d$.
In our convention all powers of $\mzv{2}$ have depth zero, hence $\MZV^0_{2k}=\Q\mzv[k]{2}$ and $\MZV^0_{2k+1}=\set{0}$.%
\footnote{%
	This definition is natural to our approach via polylogarithms (powers of $\pi$ can be generated from $\log(z)$, which has depth zero). It also simplifies theorem~\ref{thm:parity} in abolishing the need to state it modulo products.
}
There are plenty of relations between MZV.
The following well-known result, conjectured in \cite{BorweinGirgensohn:Triple}, has been proven analytically \cite{Tsumura:CombinatorialEulerZagier}, via double-shuffle relations \cite{IharaKanekoZagier:DerivationDoubleShuffle,Brown:DepthGraded,Machide:PairOfIndexSets} and from associator relations \cite{Jarossay:DepthReductions}.
\begin{theorem}[Parity for MZV]
\label{thm:parity}%
Whenever the weight $w$ and depth $d$ are of opposite parity, then $\MZV_w^d = \MZV_w^{d-1}$. 
In other words, $\mzv{n_1,\ldots,n_d}$ is a $\Q[\mzv{2}]$-linear combination of MZV of depth at most $d-1$, provided that $\abs{\vect{n}}+d$ is odd.
\end{theorem}
This theorem implies $\mzv{2k} \in \Q \mzv[k]{2}$ in depth one; $\mzv{1,2} \in \Q\mzv{3}$ and $\mzv{2,3} \in \Q \mzv{5} + \Q\mzv{2}\mzv{3}$ are examples in depth two and an explicit witness for a reduction from depth $3$ is (taken from \cite{BluemleinBroadhurstVermaseren:Datamine})
\begin{equation}\label{eq:mzv(1,5,2)}
	\mzv{1, 5, 2}
	= \tfrac{703}{875}\mzv[4]{2}-\tfrac{17}{2}\mzv{3}\mzv{5}-\tfrac{7}{10}\mzv{3,5}+2\mzv{2}\mzv[2]{3}.
\end{equation}
Note that there are two products on MZV, known as shuffle and quasi-shuffle (also called stuffle), which express a product $\mzv{\vect{n}} \mzv{\vect{m}}$ as a linear combination of MZV with integer coefficients \cite{IharaKanekoZagier:DerivationDoubleShuffle}.
For example, $\mzv{a} \mzv{b} = \mzv{a,b}+\mzv{b,a}+\mzv{a+b}$ shows that the right-hand side of \eqref{eq:mzv(1,5,2)} is indeed in $\MZV^2_8$.
In general, the products ensure that $\MZV^d_w \cdot \MZV^{d'}_{w'} \subseteq \MZV_{w+w'}^{d+d'}$.

Thinking of MZV as special values $\mzv{\vect{n}}=\Li_{\vect{n}}(1,\ldots,1)$ of multiple polylogarithms (MPL), defined by the series \cite{Goncharov:PolylogsArithmeticGeometry}
\begin{equation}
	\Li_{\vect{n}}(\vect{z})
	=
	\Li_{n_1,\ldots,n_d}(z_1,\ldots,z_d)
	\defas
	\sum_{0<k_1<\cdots<k_d} \frac{z_1^{k_1}\!\cdots z_d^{k_d}}{k_1^{n_1}\!\cdots k_d^{n_d}},
	\label{eq:def-MPL} %
\end{equation}
raises the question if theorem~\ref{thm:parity} also applies for other values of $\vect{z}$. The case when all $z_i \in \RU{N} \defas \setexp{z \in \C}{z^N=1}$ are $N$-th roots of unity has been of particular interest \cite{Goncharov:MplCyclotomyModularComplexes,Zhao:StandardRelations}, partly because such numbers occur in particle physics \cite{Broadhurst:Aufbau,Broadhurst:IrreducibleEulerSums,Broadhurst:SixthRoots}.
We set $\Li_{\emptyset} \defas 1$ in weight zero and write
\begin{equation}\label{eq:def-MZV(N)}
	\MZV^d_w(\mu_N) 
	\defas \lin_{\Q}\setexp{
		(2\pi\imag)^k \Li_{\vect{n}}(\vect{z})
	}{
		\vect{n}\in\N^r,
		\vect{z}\in\RU{N}^r, \abs{\vect{n}}+k=w, r\leq d
	}
\end{equation}
where, in contrast to \eqref{eq:MZV-depth-filtration}, $n_r = 1$ is allowed as long as $z_r\neq 1$ (this ensures convergence) and $k$ is restricted to even values in the cases $N=1,2$.

Indeed, David Broadhurst conjectured the parity theorem for alternating sums (that means $N=2$, so $z_i=\pm 1$) in \cite{Broadhurst:IrreducibleEulerSums}, but explicit proofs were only given for depth two \cite{BorweinBorweinGirgensohn:ExplicitEuler} and in low weights \cite{BluemleinBroadhurstVermaseren:Datamine}.\footnote{Jianqiang Zhao informed me of preliminary work which renders it very likely that the generating series approach of \cite{IharaKanekoZagier:DerivationDoubleShuffle} can be extended to prove the parity theorem for alternating sums.}
We like to point out though that in the meantime this conjecture has been established: It follows immediately from \cite[Th\'{e}or\`{e}me~7.2]{Deligne:GroupeFondamentalMotiviqueN}, a very deep result that completely describes the algebra of motivic alternating sums. For a simpler approach via the octagon equation, see \cite[corollary~4.2.4]{Glanois:Thesis}.

He also studied intensively the case $N=6$ in \cite{Broadhurst:SixthRoots,Broadhurst:Aufbau}, where the complexity of the arguments requires a refinement in the statement of the parity theorem: 
\begin{conjecture}[Generalized parity {\cite[section~6.2.4]{Broadhurst:SixthRoots}}]%
\label{conj:generalized-parity}%
	Let $z_1,\ldots,z_d$ be sixth roots of unity. If $\abs{\vect{n}}+d$ is odd, then the real part $\Realteil \Li_{\vect{n}}(\vect{z})$ is of depth at most $d-1$. For even $\abs{\vect{n}}+d$, the same holds for the imaginary part $\Imaginaerteil \Li_{\vect{n}}(\vect{z})$.\footnote{Again, the original formulation is slightly different (stated modulo products) but equivalent, which is easy to see via a recursive argument that shows that one of the factors must actually be $2\pi\imag$.}
\end{conjecture}
Since for roots of unity, $z_i^{-1} = \conjugate{z_i}$ is the complex conjugate, we can combine the two cases and rephrase this conjecture as
\begin{equation}
	\Li_{\vect{n}}(\vect{z}) - (-1)^{\abs{\vect{n}}_0} \Li_{\vect{n}}(1/\vect{z})
	\in \MZV^{d-1}_{\abs{\vect{n}}}(\RU{N})
	\label{eq:parity-Nth-roots}%
\end{equation}
in the case $N=6$. Here we abbreviate $\abs{\vect{n}}_0 \defas \abs{\vect{n}}-d = (n_1 + \cdots + n_d) - d$ for indices $\vect{n} \in \N^d$ and we write
\begin{equation}
	1/\vect{z}
	\defas \left( \frac{1}{z_1},\ldots,\frac{1}{z_d} \right)
	\quad\text{for vectors $\vect{z} \in \C^d$ with $z_1,\ldots,z_d \neq 0$}.
	\label{eq:def-z-inv}%
\end{equation}
Our goal is to present a simple and constructive proof of the general parity theorem~\eqref{eq:parity-Nth-roots}, valid for arbitrary values of $N$.%
\footnote{%
	While the structure of the algebra $\MZV(\RU{N})$ varies strongly with $N$ and is understood only in a few special cases \cite{Deligne:GroupeFondamentalMotiviqueN,Glanois:Descents}, the parity theorem turns out to be insensitive to these differences and applies always.%
}
Furthermore, we realized that \eqref{eq:parity-Nth-roots} is not just true at roots of unity, but in fact holds as a functional equation of multiple polylogarithms and remains valid for arbitrary values of the arguments $\vect{z}$. Concretely, in section~\ref{sec:proof} we will prove
\begin{theorem}[parity for MPL]
	\label{thm:generalized-parity}%
	For all $d\in\N$ and all indices $\vect{n} \in \N^d$, the function%
	\begin{equation}
		\PLi_{\vect{n}}(\vect{z})
		\defas \Li_{\vect{n}}(\vect{z}) - (-1)^{\abs{\vect{n}}_0} \Li_{\vect{n}}(1/\vect{z})
		\label{eq:def-PLi}%
	\end{equation}
	of the variables $\vect{z}=(z_1,\ldots,z_d)$ is of depth at most $d-1$, meaning that it can be written as a $\Q$-linear combination of the functions
	\begin{equation}
	(2\pi\imag)^{k_0} \prod_{i=1}^d \log^{k_i}(-z_i\cdots z_d)
	\prod_{i=1}^s \Li_{\vect{n}^{(i)}}(\vect{z}^{(i)})
		\label{eq:thm-generalized-parity}%
	\end{equation}
	where the indices $\vect{n}^{(i)} \in \N^{d_i}$ have total depth $d_1+\cdots+d_s < d$ and preserve the weight
	$ \abs{\vect{k}} + \sum_{i=1}^s \abs{\vect{n}^{(i)}} = \abs{\vect{n}}$.
	Each of the arguments $\vect{z}^{(i)}_{j}$ is a consecutive product $z_{\mu}z_{\mu+1}\!\cdots z_{\nu}$ for some $\mu \leq \nu$.
\end{theorem}
These functions are multivalued and we will always assume that $\vect{z}$ is in
\begin{equation}
	\C^d
	\setminus \bigcup_{1 \leq i \leq j \leq d} \setexp{\vect{z}}{z_i z_{i+1} \cdots z_j \in [0,\infty)}
	.
	\label{eq:domain}%
\end{equation}
This simply connected domain avoides the branch cuts $z_i\cdots z_d \in [0,\infty)$ of the logarithms%
\footnote{%
	With $\log(z)$ we denote the principal branch with a cut along $(-\infty,0]$, so $\log(-z)$ is analytic on $\C\setminus[0,\infty)$.
	We avoid equations like $\log(-z_1 z_2) = \log(z_1) + \log(-z_2)$ because they are not valid everywhere.
}
in \eqref{eq:def-PLi} and makes $\Li_{\vect{n}}(\vect{z})$ well-defined via analytic continuation along a straight path starting near $(0,\ldots,0)\in \C^d$ with \eqref{eq:def-MPL}, see \cite{Zhao:MPLcontinuation}.

A simple example of these functional equations in depth two is $\vect{n}=(1,2)$:
\begin{equation}\begin{split}
\label{eq:Li_1,2(z1,z2)}%
	\PLi_{1,2}(z_1,z_2)
	&=
	\Li_{1,2}\left(\tfrac{1}{z_1},\tfrac{1}{z_2}\right)
	+ \Li_{1,2}(z_1,z_2)
	\\
	&=
	\Li_3(z_1)+2\Li_3(z_2)-\Li_3(z_1 z_2)
	+2\mzv{2}\log(-z_2)
	\\
	&-\log(-z_1 z_2)\left[
		\Li_2(z_2)+\Li_2(z_1)+\mzv{2}
		+\tfrac{1}{2} \log^2(-z_2)
	\right]
	\\
	&+\tfrac{1}{2}\Li_1(z_1)\left[
		\log^2(-z_1 z_2)
		-\log^2(-z_2)
	\right]
	+\tfrac{1}{3}\log^3(-z_2)
	.
\end{split}\end{equation}
We will show how these functional equations can be computed explicitly for arbitrary indices $\vect{n}$. To illustrate this, we derive closed expressions (for arbitrary weights) in depths two and three, see \eqref{eq:PLi-depth-2} and \eqref{eq:PLi-depth-3}.

A striking feature of theorem~\ref{thm:generalized-parity} is the absence of any transcendental numbers in the functional equation, except for powers of $2\pi\imag$ in \eqref{eq:thm-generalized-parity}. In the limit when all $z_i$ approach roots of unity, we therefore obtain (see lemma~\ref{lem:z_d->1})
\begin{corollary}
	\label{cor:generalized-parity}%
	The generalized parity \eqref{eq:parity-Nth-roots} holds for arbitrary $N\in\N$.
\end{corollary}
In our example, the limit $z_1 \rightarrow 1$ in \eqref{eq:Li_1,2(z1,z2)} is smooth%
\footnote{%
	As the left-hand side of \eqref{eq:Li_1,2(z1,z2)} is analytic at $z_1=1$ (the only singularities are $z_i=0$, $z_2=1$ and $z_1 z_2=1$), the non-analytic contributions from $\Li_3(z_1)$, $\Li_2(z_1)$ and $\Li_1(z_1)$ in the right-hand side cancel each other.
}
and gives
\begin{equation*}
	\Li_{1,2}\left(1,\tfrac{1}{z_2}\right)
	+ \Li_{1,2}(1,z_2)
	=
	\Li_3(z_2) + \mzv{3}-\log(-z_2)\Li_2(z_2)-\tfrac{1}{6}\log^3(-z_2)
	.
\end{equation*}
Subsequent limits like $z_2\rightarrow 1,-1,\imag$ yield explicit depth-reductions such as
\begin{align*}
	2 \mzv{1,2} &= 2\mzv{3}, \\
	2 \Li_{1,2}(1,-1) &= \mzv{3} + \Li_3(-1) = \tfrac{1}{4} \mzv{3} \quad\text{and} \\
	2 \Realteil \Li_{1,2}(1,\imag) &= \mzv{3} + \Li_3(\imag) - \tfrac{1}{48} \imag\pi^3 +\tfrac{1}{2}\imag\pi \Li_2(\imag) = \tfrac{29}{32} \mzv{3} - \tfrac{1}{2} \pi \Imaginaerteil \Li_2(\imag),
\end{align*}
where we simplified the result with $\Realteil \Li_3(\imag) = - \frac{3}{32} \mzv{3}$ \cite{Lewin:PolylogarithmsAssociatedFunctions}.%
\footnote{%
	Note that in the last case, the imaginary parts $\Imaginaerteil \Li_3(\imag)$ and $\frac{\pi}{2}\Realteil \Li_2(\imag)$ must cancel with $-\frac{\pi^3}{48}$, which is indeed easily verified with \eqref{eq:polylog-inversion}.%
}%
As a consequence of our closed formulas for $\PLi_{\vect{n}}(\vect{z})$ in small depths, we also obtain explicit formulas for special cases like MZV. In depth two such a formula is well-known \cite{BorweinBorweinGirgensohn:ExplicitEuler,Zagier:232,HuardWilliamsZhang:Tornheim}, see \eqref{eq:MZV-depth2}, but in depth three and even weight $w=n_1+n_2+n_3$, the reduction
\begin{equation}\begin{split}\label{eq:MZV-depth-3}
	\mzv{n_1,n_2,n_3}
	&=-\frac{1}{2}\Big(
		\mzv{n_1+n_2,n_3}
		+\mzv{n_1,n_2+n_3}
		+\mzv{w}
	\Big)
	\\&
	+(-1)^{n_1} \sum_{\mathclap{\substack{\mu+\nu+2s=w\\\mu\geq n_2,\nu \geq n_3}}}
		\mzv{2s}
		\binom{\mu-1}{n_2-1}\binom{\nu-1}{n_3-1}
		\mzv{\mu,\nu}
	\\&
	+(-1)^{n_2} \sum_{\mathclap{\substack{\mu+\nu+2s=w\\ \mu\geq n_3, \nu > n_1}}}
		\mzv{2s}
		\binom{\mu-1}{n_3-1}\binom{\nu-1}{n_1-1}
		\mzv{\mu}\mzv{\nu}(-1)^{\mu}
	\\&
	+(-1)^{n_3} \sum_{\mathclap{\substack{\mu+\nu+2s=w\\\mu\geq n_2,\nu>n_1}}}
		\mzv{2s}
		\binom{\mu-1}{n_2-1}\binom{\nu-1}{n_1-1}
		\mzv{\mu,\nu} 
	\\&
	-(-1)^{n_3} \sum_{\mathclap{\substack{n_1+\nu+2s=w\\\nu>n_2}}}
		\mzv{2s}
		\binom{\nu-1}{n_2-1}
		\Big(\mzv{n_1,\nu}+\mzv{n_1+\nu}\Big)
\end{split}\end{equation}
for arbitrary $\vect{n}\in\N^3$, $n_3>1$ appears to be new. Note that a related formula, valid for $n_1, n_2 > 1$, was given in \cite[Theorem~5.2]{KomoriMatsumotoTsumura:MZVfromRootSystems}. Analogous depth reductions of polylogarithms at arbitrary roots of unity can be obtained from \eqref{eq:PLi-depth-2} and \eqref{eq:PLi-depth-3} as demonstrated in section~\ref{sec:depth-2}.
Furthermore, such explicit formulas can in principle be derived for arbitrary depth by continuation of the recursive algorithm that we will describe in section~\ref{sec:proof}.

Notice that except for the occurrence of $\mzv{0}=-1/2$, all coefficients in \eqref{eq:MZV-depth-3} are integers. Indeed, we will actually prove a stronger version of theorem~\ref{thm:generalized-parity} with integer coefficients (theorem~\ref{thm:generalized-parity-MPLs}), by replacing the logarithms with Bernoulli polynomials. In the case of MZV, this strengthens theorem~\ref{thm:parity} to the statement that ($k=0$ is allowed here)
\begin{equation}\label{eq:parity-MZV-integer}
	\mzv{n_1,\ldots,n_d}
	\in 
	\sum_{\substack{\abs{\vect{m}}+2k=\abs{\vect{n}}\\\vect{m}\in\N^r,\ r<d}} 
		\Z \mzv{2k}\mzv{\vect{m}}
\end{equation}
whenever $\abs{\vect{n}}_0=\abs{\vect{n}}-d$ is odd. This is not obvious in the example \eqref{eq:mzv(1,5,2)} above, but the representation we get from \eqref{eq:MZV-depth-3}  makes \eqref{eq:parity-MZV-integer} manifest:
\begin{equation}\label{eq:mzv(1,5,2)-integer}\begin{split}
	\mzv{1,5,2}
	&=
	7\mzv{8}
	+3\mzv{2}\mzv{6}
	-5\mzv{4}\mzv{2}^2
	+2\mzv{2}\mzv[2]{3}
	\\&\quad
	-3\mzv{3}\mzv{5}
	+7\mzv{1,7}+\frac{3\mzv[2]{4}+\mzv{5,3}-\mzv{6,2}}{2}
	.
\end{split}\end{equation}
An integrality statement also exists for other roots of unity (see remark~\ref{rem:Nth-roots-integer}). 

As a supplement to this article, we provide a list of the functional equations for all $\PLi_{\vect{n}}(\vect{z})$ of weight $\abs{\vect{n}} \leq 6$ in the file \Filename{feqs}. These were computed using {\MapleTM} with the script \Filename{parity.mpl} (also attached) which can be used to compute higher weight equations.{\MapleNote}

\paragraph{Acknowledgements}
My interest in the parity conjecture arose from the computation of a Feynman integral which evaluates to MPL at sixth roots of unity, where this conjecture played a crucial role \cite{Panzer:PhD,Broadhurst:Aufbau}.%
\footnote{%
	For this particular case an independent proof of the parity theorem was given in \cite[equation~(3.5.13)]{Panzer:PhD}.
}
I thank David Broadhurst, Dirk Kreimer and Oliver Schnetz for suggesting this project in the first place.
Moreover I am grateful to Oliver for the collaboration \cite{PanzerSchnetz:Phi4Coaction}, which required a better understanding of parity.
Also I thank the ESI Vienna for hospitality during the workshop \href{http://www.esi.ac.at/activities/events/2015/the-interrelation-between-mathematical-physics-number-theory-and-non-commutative-geometry}{``The interrelation between mathematical physics, number theory and noncommutative geometry''}, where exciting discussions with David sparked my interest in the general case.

Furthermore I am indebted to Francis Brown, Claire Glanois, David Jarossay and Jianqiang Zhao for helpful feedback on topics related to the parity theorem.
Takashi Nakamura kindly pointed out is work \cite{Nakamura:SimpleLerchTornheim} to me and helped to compare his results with \eqref{eq:PLi-depth-2}.

Finally I am very grateful to a referee for careful reading and many valuable suggestions which improved the clarity of this paper.

This work was carried out at the \href{http://www.ihes.fr/}{IH\'{E}S} with support from the ERC grant 257638 via the \href{http://www.cnrs.fr/}{CNRS} and at \href{https://www.asc.ox.ac.uk/}{All Souls College}.

\section{Proof of the parity theorem}
\label{sec:proof}
Our starting point is the well-known formula for depth $d=1$, 
\begin{equation}
	\PLi_n(z) =
	\Li_n(z) + (-1)^n \Li_n(1/z)
	=-\frac{(2\pi\imag)^n}{n!} B_n \left( \frac{1}{2} + \frac{\log(-z)}{2\pi\imag} \right),
	\label{eq:polylog-inversion}%
\end{equation}
which relates \emph{classical} polylogarithms of inverse arguments via the Bernoulli polynomials $B_n(x)$ that are generated by
\begin{equation}
	\frac{t e^{xt}}{e^t -1}
	= \sum_{n=0}^{\infty} B_n(x) \frac{t^n}{n!}
	.
	\label{eq:bernoulli-polynomials}%
\end{equation}
Equation~\eqref{eq:polylog-inversion} holds for all $z \in \C \setminus [0,\infty)$ and goes back to \cite{Jonquiere:Thesis}.%
\footnote{%
	Note that both sides of \eqref{eq:polylog-inversion} are analytic on this domain since we always consider the principal branches.
} 
The inductive proof given in \cite{Jonquiere:Thesis} exploits that both $(z\partial_z) \PLi_{n+1}(z) = \PLi_{n}(z)$ and $(z\partial_z) \ber{n+1}{z} = \ber{n}{z}$ fulfill the same differential equation, where
\begin{equation}
	\ber{n}{z}
	\defas 
	\frac{(2\pi\imag)^n}{n!} B_n \left( \frac{1}{2} + \frac{\log(-z)}{2\pi\imag} \right)
	= (-1)^n \ber{n}{1/z}
	\label{eq:ber}%
\end{equation}
denotes the opposite of the right-hand side of \eqref{eq:polylog-inversion}. After checking
\begin{equation*}
	\PLi_1(z) 
	= -\log(1-z)+\log(1-1/z) 
	= - \log(-z)
	= - \ber{1}{z}
\end{equation*}
for $z\in \C\setminus[0,\infty)$ to start the induction, it suffices to verify one boundary condition for each $n\geq 2$ to fix the constants of integration. The choice $z\rightarrow 1$ yields $\PLi_{2n}(z) \rightarrow 2 \mzv{2n}$ in even weight, consistent with 
\begin{equation}\label{eq:even-mzv}
	2\mzv{2n}
	= -\frac{(2\pi\imag)^{2n}}{(2n)!} B_{2n}
	= -\ber{2n}{1}
\end{equation}
for the Bernoulli numbers $B_{2n}\defas B_{2n}(0) = B_{2n}(1)$.\footnote{%
	This symmetry of Bernoulli polynomials cancels the ambiguity $\log(-z) \rightarrow \pm \imag \pi$ depending on whether $z\rightarrow 1$ is approached from the upper or lower half-planes.%
}
Finally, in odd weight, both $\PLi_{2n+1}(1)$ and $\ber{2n+1}{1}$ vanish for $n \geq 1$.

In order to extend \eqref{eq:polylog-inversion} to MPL, we use differentiation to reduce the weight, apply the theorem inductively, and integrate thereafter. This recursive approach rests on explicit differentials easily obtained from \eqref{eq:def-MPL}: For $n_1>1$,
\begin{equation}
	\partial_{z_1} \Li_{\vect{n}}(\vect{z})
	=
	\frac{1}{z_1} \Li_{\vect{n}-\vect{e}_1}(\vect{z})
	\quad\text{and}\quad
	\partial_{z_1} \Li_{\vect{n}}(1/\vect{z})
	=
	-\frac{1}{z_1} \Li_{\vect{n}-\vect{e}_1}(1/\vect{z})
	\label{eq:diff-MPL}%
\end{equation}
where $\vect{e}_1 = (1,0,\ldots,0)$ denotes the first unit vector, while for $\vect{n}=(1,\vect{n}')$ we set $\vect{z}' = (z_2,\ldots,z_d)$ and $\vect{z}'' = (z_1z_2,z_3,\ldots,z_d)$ such that
\begin{align}
	\partial_{z_1} \Li_{\vect{n}}(\vect{z})
	&= \frac{\Li_{\vect{n}'}(\vect{z}')}{1-z_1}
	-\frac{\Li_{\vect{n}'}(\vect{z}'')}{z_1(1-z_1)}
	\quad\text{and}
	\label{eq:diff-MPL-one}%
	\\
	\partial_{z_1} \Li_{\vect{n}}(1/\vect{z})
	&=\frac{\Li_{\vect{n}'}(1/\vect{z}')}{z_1(1-z_1)}
	-\frac{\Li_{\vect{n}'}(1/\vect{z}'')}{1-z_1}
	.
	\label{eq:diff-MPL-one-inverse}%
\end{align}
For example, in depth two these differentials, together with \eqref{eq:polylog-inversion}, show
\begin{equation}\label{eq:diff-PLi_1,n}
	\partial_{z_1} \PLi_{1,n}(z_1, z_2)
	= \frac{\ber{n}{z_1 z_2}-\ber{n}{z_2}}{1-z_1}
	+ \frac{(-1)^n\Li_n(1/z_2) - \Li_n(z_1 z_2)}{z_1}
\end{equation}
and we could apply \eqref{eq:polylog-inversion} once more to further replace $(-1)^n\Li_n(1/z_2)$ with $-\Li_n(z_2) - \ber{n}{z_2}$. In general, we will need the following class of functions:
\begin{definition}\label{def:MPLs}%
	Given a vector $\vect{z} = (z_1,\ldots,z_N)$ of variables, we abbreviate consecutive products with $z_{i,j} \defas \prod_{k=i}^j z_k$ where $i\leq j$.
	We define a $\Z$-module by
	\begin{equation}
		\MPLs^D_w(\vect{z})
		\defas
		\lin_{\Z} \setexp{
			\prod_{i=1}^N \ber{k_i}{z_{i,N}}
			\prod_{i=1}^s
			\Li_{\vect{n}^{(i)}}(\vect{z}^{(i)})
		}{ 
			\vect{k} \in \N_0^{N},
			s \in \N_0
		},
		\label{eq:def-MPLs}%
	\end{equation}
	where for each $1\leq i \leq s$, the integer indices $\vect{n}^{(i)} \in \N^{d_i}$ and the arguments $\vect{z}^{(i)}$ are vectors of size $d_i > 0$ such that the total depth is $d = d_1+\cdots+d_s \leq D$ and the total weight is
	$\abs{\vect{k}}+\sum_{i=1}^s \abs{\vect{n}^{(i)}} = w$.
	As arguments $\vect{z}^{(i)}$ we allow only consecutive products which are disjoint and totally ordered. Concretely, we mean that if we read all of the $d$ arguments in $\Li_{\vect{n}^{(1)}}(\vect{z}^{(1)}) \cdots \Li_{\vect{n}^{(s)}}(\vect{z}^{(s)})$ from left to right,
	\begin{equation*}
		(\vect{z}^{(1)},\ldots,\vect{z}^{(s)}) 
		= (z_{i_1,j_1},\ldots, z_{i_{d},j_{d}}),
	\end{equation*}
	then $1 \leq i_1 \leq j_1 < i_2 \leq j_2 < \cdots < i_{d} \leq j_{d} \leq N $.
\end{definition}
The constraints on the arguments exclude products like
\begin{equation*}
	\Li_{n_1}(z_1) \Li_{n_2}(z_1),
	\quad
	\Li_{n_1,n_2}(z_1, z_3) \Li_{n_3}(z_2)
	\quad\text{and also}\quad
	\Li_{n_1,n_2,n_3}(z_1,z_3,z_2)
\end{equation*}
from \eqref{eq:def-MPLs}. In particular note that each variable $z_i$ appears in at most one argument and $z_1$ can only appear in the first position.
For example, in the case of two variables $\vect{z}=(z_1,z_2)$ the explicit generators up to depth two are
\begin{align*}
	\MPLs^0_w(\vect{z})
	&= 
	\bigoplus_{\mathclap{k_1 + k_2 = w}} \Z \ber{k_1}{z_1 z_2} \ber{k_2}{z_2}
	\subset \Q[2\pi\imag,\log(-z_1 z_2),\log(-z_2)],
	\\ 
	\MPLs^1_w(\vect{z})
	&= \MPLs^0_w(\vect{z}) \oplus \bigoplus_{\mathclap{k+n=w}}
		\MPLs^0_{k}(\vect{z}) \tp
	\lin_{\Z}
	\set{
		\Li_{n}(z_1),
		\Li_{n}(z_2),
		\Li_{n}(z_1 z_2)
	}
	\quad\text{and}
	\\
	\MPLs^2_w(\vect{z})
	&= \MPLs^1_w(\vect{z}) \oplus \bigoplus_{\mathclap{k+n_1+n_2=w}} \MPLs^{0}_k(\vect{z}) \tp \lin_{\Z} \set{
		\Li_{n_1}(z_1) \Li_{n_2}(z_2),
		\Li_{n_1, n_2}(z_1,z_2)
	}.
\end{align*}
All of the generators \eqref{eq:def-MPLs} are in fact linearly independent, which follows from their representation as iterated integrals and can be proved inductively through differentiation \cite{Panzer:PhD}. So indeed we can write direct sums; however, this linear independence will not play any role in the sequel.

In theorem~\ref{thm:generalized-parity} we did not have to bother about the fine structure of the arguments $\vect{z}^{(i)}$; knowing that they are products of the original variables is enough to conclude the reduction in depth for roots of unity in \eqref{eq:parity-Nth-roots} (corollary~\ref{cor:generalized-parity}).
But for our proof it is essential to keep track of the arguments very closely.
The following result shows that $\MPLs^D_w(\vect{z})$ contains all primitives we need to integrate $\partial_{z_1} \PLi_{\vect{n}}(\vect{z})$, like the right-hand side of \eqref{eq:diff-PLi_1,n}:
\begin{lemma}
	\label{lemma:primitives}%
	Every element of $\frac{1}{z_1} \MPLs_{w}^{D}(\vect{z})$ has a $\partial_{z_1}$-primitive in $\MPLs_{w+1}^{D}(\vect{z})$ and
	every element of $\frac{1}{1-z_1} \big[\MPLs_{w}^{D}(\vect{z}')+\MPLs_{w}^{D}(\vect{z}'')\big]$ has a $\partial_{z_1}$-primitive in $\MPLs_{w+1}^{D+1}(\vect{z})$.\footnote{%
	As before, $\vect{z}'=(z_2,\ldots,z_N)$ and $\vect{z}'' = (z_1 z_2, z_3, \ldots, z_N)$.
}%
\end{lemma}
\begin{proof}
	By definition~\ref{def:MPLs}, each basis function in \eqref{eq:def-MPLs} has at most two factors which depend on $z_1$: $\ber{k}{z_{1,N}}$ and the first factor $\Li_{\vect{n}^{(1)}}(\vect{z}^{(1)})$. All other factors are constants with respect to $z_1$ and can be ignored here.
	We proceed by giving explicit formulas for the primitives which can be verified easily via differentiation using \eqref{eq:diff-MPL} and \eqref{eq:diff-MPL-one}.

	If the $z_1$-dependence only comes from logarithms, we can use
	\begin{align}
		\frac{\ber{k}{z_{1,N}}}{z_1}
		&=
		\partial_{z_1} \ber{k+1}{z_{1,N}}
		\quad\text{and}
		\label{eq:primitive-ber-0}\\
		\frac{\ber{k}{z_{1,N}}}{1-z_1}
		&=
		\partial_{z_1} \sum_{\mu=0}^k (-1)^{\mu}
		\Li_{1+\mu}(z_1) \ber{k-\mu}{z_{1,N}}
		.
		\label{eq:primitive-ber-1}%
	\end{align}
	Otherwise we have to integrate a function of the form $\ber{k}{z_{1,N}} \Li_{\vect{n}}(z_{1,j}, \vect{y})$ with some index $1\leq j \leq N$ and a sequence $\vect{y}$ of $z_1$-independent arguments. Integrating $\dd z_1/{z_1}$ does not change the depth:
	\begin{equation}
		\frac{\ber{k}{z_{1,N}} \Li_{\vect{n}}(z_{1,j}, \vect{y})}{z_1}
		=
		\partial_{z_1} \sum_{\mu=0}^k (-1)^{\mu} \ber{k-\mu}{z_{1,N}} 
		\Li_{\vect{n}+(1+\mu)\vect{e}_1}(z_{1,j}, \vect{y}) 
		.
		\label{eq:primitive-MPL-0}%
	\end{equation}
	If the denominator is $1-z_1$, we consider a function in $\MPLs^D_w(\vect{z}'')=\MPLs^D_w(z_1 z_2,\ldots)$ because the case of $\MPLs^D_w(\vect{z}')$ was dealt with in \eqref{eq:primitive-ber-1}. Therefore we must have $j\geq 2$ and find a primitive which increases the depth by one:
	\begin{multline}
		\frac{\ber{k}{z_{1,N}}\Li_{\vect{n}}(z_{1,j},\vect{y})}{1-z_1}
		=\partial_{z_1}
		\sum_{\mu=0}^{k} (-1)^{\mu} \ber{k-\mu}{z_{1,N}} \Big[
			\Li_{1+\mu}(z_1)\Li_{\vect{n}}(z_{2,j}, \vect{y})
		\\
			-\Li_{1+\mu,\vect{n}}(z_1, z_{2,j}, \vect{y})
			-\Li_{\vect{n}+(1+\mu)\vect{e}_1}(z_{1,j}, \vect{y})
		\Big]
		.
		\qedhere
		\label{eq:primitive-MPL-1}%
	\end{multline}
\end{proof}
Applied to the right-hand side of example \eqref{eq:diff-PLi_1,n}, this construction yields\footnote{For brevity, we do not rewrite $(-1)^n \Li_n(1/z_2) = -\Li_n(z_2) - \ber{n}{z_2}$ explicitly.}
\begin{equation}\begin{split}\label{eq:PLi_1,n-primitive}
	\MPLs_{n+1}^1(\vect{z}) \ni
	F
	&\defas \sum_{\mu=0}^{n} (-1)^{\mu} \Li_{1+\mu}(z_1) \ber{n-\mu}{z_1 z_2}
	- \Li_1(z_1) \ber{n}{z_2}
	\\
	&\quad - \Li_{n+1}(z_1 z_2)
	+(-1)^n \Li_n(1/z_2) \ber{1}{z_1 z_2}
\end{split}\end{equation}
such that $\partial_{z_1} \big[\PLi_{1,n}(\vect{z}) - F\big] = 0$. 
The final ingredient is a method to determine the constant of integration (it depends on $z_2$ in this case), which we will fix by considering the limit $z_1 \rightarrow 0$. 
Note that according to \eqref{eq:def-MPL}, all $z_1$-dependent MPL $\Li_{\vect{n}}(z_{1,j},\ldots)$ vanish in this limit, so logarithms $\ber{k}{z_{1,N}}$ provide the only divergences. Let us write $f \sim g$ if $\lim_{z_1\rightarrow 0} (f-g) = 0$.
\begin{lemma}\label{lem:reglim}
	Given $\vect{n}=(n_1,\vect{n}') \in \N^d$, the behaviour of $\Li_{\vect{n}}(1/\vect{z})$ as $z_1\rightarrow 0$ is
	\begin{equation}\label{eq:reglim}
		\Li_{\vect{n}}(1/\vect{z})
		\sim (-1)^{1+n_1} \sum_{\mathclap{\abs{\vect{k}}=n_1}}
		\ber{k_1}{z_{1,d}}
		\Li_{\vect{n}'+\vect{k}'}(1/\vect{z}')
		\prod_{\mu=2}^d \binom{n_{\mu}-1+k_{\mu}}{k_{\mu}}
		.
	\end{equation}
\end{lemma}
\begin{proof}
	We exploit the representation from \cite{Goncharov:PolylogsArithmeticGeometry,Goncharov:MplCyclotomyModularComplexes}, which expresses MPL as iterated integrals \cite{Chen:AlgebrasOfII} of the differential forms $\letter{\sigma} \defas \dd t/(t-\sigma)$,
\begin{equation}\label{eq:MPL-as-Hlog} \!\!
	(-1)^d \Li_{\vect{n}}(1/\vect{z})
	= \!\!
	\int_0^1 \!
		\letter{0}^{n_d-1} \letter{z_d}
		\letter{0}^{n_{d-1}-1} \letter{z_{d-1} z_d}
		\!\cdots
		\letter{0}^{n_{2}-1}\letter{z_{2,d}} 
		\letter{0}^{n_1-1} \letter{z_{1,d}}
	.
\end{equation}
These are called hyperlogarithms \cite{LappoDanilevsky:CorpsRiemann,Poincare:GroupesEquationsLineaires} and defined by
\begin{equation*}\label{eq:Hlog}
	\int_0^1 \letter{\sigma_r} \!\cdots\letter{\sigma_1}
	= \int_{0}^{1}\frac{\dd t_r}{t_r-\sigma_r} \int_{0}^{t_r}\frac{\dd t_{r-1}}{t_{r-1}-\sigma_{r-1}} \cdots \int_0^{t_2} \frac{\dd t_1}{t_1 - \sigma_1}.
\end{equation*}
By \cite{Ree:LieShuffles} they are multiplicative,
$\int_0^1 u \cdot \int_0^1 v = \int_0^1 (u \shuffle v)$,
with respect to the commutative shuffle product defined recursively by
\begin{equation}
	(\letter{\sigma} u) \shuffle (\letter{\tau}  v)
	= \letter{\sigma}(u \shuffle (\letter{\tau} v))
	+ \letter{\tau}((\letter{\sigma} u) \shuffle v)
	\label{eq:shuffle-product}%
\end{equation}
for arbitrary (including empty) words $u$ and $v$ in the tensor algebra generated by the forms $\letter{\sigma}$.
We apply the simple combinatorial identity \cite[lemma~3.2.5]{Panzer:PhD}
\begin{equation}
	u\letter{\tau} \letter{\sigma_1}\!\cdots\letter{\sigma_r}
	= \sum_{k=0}^r (-1)^{k}\left[ u \shuffle \letter{\sigma_{k}}\!\cdots \letter{\sigma_{1}} \right] \letter{\tau}
	\shuffle \letter{\sigma_{k+1}}\!\cdots\letter{\sigma_r}
	\label{eq:shuffle-reg}%
\end{equation}
to $u=\letter{0}^{n_d-1}\letter{z_d}\!\cdots \letter{z_{3,d}}\letter{0}^{n_2-1}$ with $\tau=z_{2,d}$ and thus rewrite \eqref{eq:MPL-as-Hlog} as
\begin{equation*}\begin{split}
	(-1)^d \Li_{\vect{n}}(1/\vect{z})
	&=\sum_{k=0}^{n_1-1} (-1)^k \int_0^1 (u \shuffle \letter{0}^k) \letter{z_{2,d}} \int_0^1 \letter{0}^{n_1-1-k} \letter{z_{1,d}}
	\\&\quad
	+(-1)^{n_1} \int_0^1 (u \shuffle \letter{z_{1,d}}\letter{0}^{n_1-1}) \letter{z_{2,d}}
	.
\end{split}\end{equation*}
The last term is continuous at $z_1 \rightarrow 0$ (so that we can just substitute $z_1=0$), and inside the sum we use \eqref{eq:MPL-as-Hlog} and \eqref{eq:polylog-inversion} to rewrite
$
	\int_0^1 \letter{0}^{n_1-1-k}\letter{z_{1,d}}
	=
	-\Li_{n_1-k}(1/z_{1,d})
	\sim
	(-1)^{n_1-k}\ber{n_1-k}{z_{1,d}}
$. 
This proves
\begin{equation*}
	\Li_{\vect{n}}(1/\vect{z})
	\sim (-1)^{n_1+d}\sum_{k=0}^{n_1} 
	\ber{k}{z_{1,d}}
	\int_0^1 \left( u \shuffle \letter{0}^{n_1-k} \right)\letter{z_{2,d}}
\end{equation*}
and we conclude by expanding the products $u \shuffle \letter{0}^{n_1-k}$:
The coefficients of individual words are counted by the binomial coefficients in \eqref{eq:reglim} and we rewrite their hyperlogarithms via \eqref{eq:MPL-as-Hlog} as MPL with depth $d-1$. 
\end{proof}
\begin{remark}\label{rem:order-of-limits}
	All terms in \eqref{eq:reglim} have depth $d-1$ and no MZV occur in the expansion of $\Li_{\vect{n}}(\vect{z})$ at $z_1\rightarrow 0$ (except for powers of $\pi$ in \eqref{eq:ber}, to which we assign depth zero).\footnote{%
		In the language of regularized limits (which annihilate logarithmic divergences, see \cite{Brown:TwoPoint,Panzer:PhD}), the essence of lemma~\ref{lem:reglim} is that
	$
		\AnaReg{z_d}{\infty}\cdots\AnaReg{z_1}{\infty} \Li_{\vect{n}}(\vect{z})
		\in (2\pi\imag)^{\abs{\vect{n}}}\Q
	$
	does not involve any MZV.
}
	This is crucial for our proof and the reason why we set up our recursion with respect to the first variable $z_1$: While the differential behaviour is analogous if we considered $z_d$ instead, the limiting behaviour $z_d\rightarrow 0$ is more complicated. An analogue of lemma~\ref{lem:reglim} for this case is not obvious.
\end{remark}
Continuing our example from above, we obtain
\begin{equation*}
	(-1)^n\PLi_{1,n}(\vect{z})
	\sim \Li_{1,n}(1/\vect{z})
	\sim n \Li_{n+1}(1/z_2) + \ber{1}{z_1 z_2} \Li_{n}(1/z_2)
	.
\end{equation*}
The divergent contribution $F \sim (-1)^n \ber{1}{z_1 z_2} \Li_n(1/z_2)$ is already contained in \eqref{eq:PLi_1,n-primitive}, but we now also know the constant of integration and find\footnote{%
	To check that this reproduces \eqref{eq:Li_1,2(z1,z2)} in the special case $n=2$, use \eqref{eq:polylog-inversion} to rewrite $\Li_2(1/z_2)$ and $\Li_3(1/z_2)$ and substitute $\ber{1}{z}=\log(-z)$, $\ber{2}{z} = \frac{1}{2}\log^2(-z)+\mzv{2}$ and $\ber{3}{z} = \frac{1}{6} \log^3(-z) + \mzv{2} \log(-z)$.%
}
\begin{align}
	\PLi_{1,n}(\vect{z})
	&= F + n(-1)^{n} \Li_{n+1}(1/z_2)
	\label{eq:PLi_1,n}\\
	&= \sum_{\mu=0}^{n} (-1)^{\mu} \Li_{1+\mu}(z_1) \ber{n-\mu}{z_1 z_2}
	- \Li_1(z_1) \ber{n}{z_2}
	- \Li_{n+1}(z_1 z_2)
	\nonumber\\
	&\quad 
	+(-1)^n \big[n \Li_{n+1}(1/z_2)+ \Li_n(1/z_2) \ber{1}{z_1 z_2} \big]
	\in \MPLs_{n+1}^1(z_1, z_2). \nonumber
\end{align}
Finally, we can now prove our parity theorem~\ref{thm:generalized-parity}, formulated inside the very explicitly described module $\MPLs^D_w(\vect{z})$ of multiple polylogarithms from \eqref{eq:def-MPLs}.
\begin{theorem}
	For all indices $\vect{n} \in \N^d$, the combination $\PLi_{\vect{n}}(\vect{z})$ of MPL from \eqref{eq:def-PLi} is of depth at most $d-1$: $\PLi_{\vect{n}}(\vect{z}) \in \MPLs_{\abs{\vect{n}}}^{d-1}(\vect{z})$.
	\label{thm:generalized-parity-MPLs}%
\end{theorem}
\begin{proof}
	We proceed by recursion over the weight $w=\abs{\vect{n}}$. The depth one case is \eqref{eq:polylog-inversion}, so let us assume $d \geq 2$. If $n_1 > 1$, we use \eqref{eq:diff-MPL} to find
\begin{equation}\label{eq:diff-PLi}
	\partial_{z_1} \PLi_{\vect{n}}(\vect{z})
	= \frac{1}{z_1} \PLi_{\vect{n}-\vect{e}_1}(\vect{z})
\end{equation}
which has depth at most $d-1$ (by induction over the weight), and integrating $z_1$ does not increase the depth according to lemma~\ref{lemma:primitives}. What remains to be checked is that the correct constant of integration also lies in $\MPLs_{w}^{d-1}(\vect{z})$. This is evident from \eqref{eq:reglim}.

Now consider the case $n_1 = 1$. From \eqref{eq:diff-MPL-one} and \eqref{eq:diff-MPL-one-inverse} we find the key formula
\begin{equation}
	\partial_{z_1} \PLi_{\vect{n}}(\vect{z})
	=
	\frac{\PLi_{\vect{n}'}(\vect{z}')-\Li_{\vect{n}'}(\vect{z}') - \Li_{\vect{n}'}(\vect{z}'')}{z_1}
	+\frac{\PLi_{\vect{n}'}(\vect{z}') - \PLi_{\vect{n}'}(\vect{z}'')}{1-z_1}
	,
	\label{eq:diff-PLi-one}%
\end{equation}
where $\vect{z}=(z_1,\vect{z}')$ and $\vect{z}'' = (z_1 z_2, z_3, \ldots, z_d)$.
It is crucial here that the signs work out such that the MPL with denominator $1-z_1$ can be grouped together as $\PLi_{\vect{n}'}(\vect{z}')$ and $\PLi_{\vect{n}'}(\vect{z}'')$, because we know by induction that these have depth at most $d-2$. We conclude that
\begin{equation*}
	\partial_{z_1} \PLi_{\vect{n}}(\vect{z})
	\in
	\frac{1}{z_1} \MPLs_{w-1}^{d-1}(\vect{z})
	+
	\frac{1}{1-z_1}\left(\MPLs_{w-1}^{d-2}(\vect{z}') + \MPLs_{w-1}^{d-2}(\vect{z}'') \right)
\end{equation*}
and hence find a $\partial_{z_1}$-primitive $f \in \MPLs_{w}^{d-1}(\vect{z})$ such that $\partial_{z_1} \PLi_{\vect{n}}(\vect{z}) = \partial_{z_1} f$ according to lemma~\ref{lemma:primitives}.
From \eqref{eq:reglim} we know that $\PLi_{\vect{n}}(\vect{z}) - f \sim g \in \MPLs^w_{d-1}(\vect{z}')$ and thus finally conclude that $\PLi_{\vect{n}}(\vect{z})=f+g \in \MPLs_{d-1}^{w}(\vect{z})$.
\end{proof}
Our algorithmic proof can be used to compute explicit representations of $\PLi_{\vect{n}}(\vect{z})$ as an element of $\MPLs_{\abs{\vect{n}}}^{d-1}(\vect{z})$ like \eqref{eq:PLi_1,n}. In sections~\ref{sec:depth-2} and \ref{sec:depth-3} we will demonstrate this by deriving closed formulas in depths two and three.
\begin{remark}
	An alternative interpretation of MPL are iterated integrals in several variables, $\Li_{\vect{n}}(\vect{z}) = \int_{\vect{0}}^{\vect{z}} w$, where $w$ is a linear combination of words in the differential forms $\dd \log (z_i)$ and $\dd \log(1-z_{i,j})$ \cite{BognerBrown:GenusZero,Zhao:MPLcontinuation}.
	The depth equals the maximum number of forms of the second type that appear in a single word.
	Hence the pullback
\begin{equation*}
	\Inversion^{\ast} \dd \log (z_i)
	= - \dd \log(z_i)
	,\quad
	\Inversion^{\ast} \dd \log(1-z_{i,j})
	= \dd \log(1-z_{i,j}) - \sum_{\mathclap{i\leq k \leq j}}\dd \log(z_k)
\end{equation*}
	under $\Inversion(\vect{z}) \defas 1/\vect{z}$ maps $w$ to $(-1)^{\abs{\vect{n}}_0} w$, up to terms of lower depth. Thus $\PLi_{\vect{n}}(\vect{z})$ reduces to lower depth MPL, because
$
	\Li_{\vect{n}}(1/\vect{z}) 
	= \int_0^{1/\vect{z}} w 
	= \int_{\vect{\infty}}^{\vect{z}} \Inversion^{\ast}(w)
$
equals $\int_{\vect{0}}^{\vect{z}} \Inversion^{\ast}(w)$ plus products of lower depth MPL with MZV, according to the path-concatenation formula \cite{Chen:AlgebrasOfII}.
	However, we do not want to contaminate theorem~\ref{thm:generalized-parity-MPLs} with such MZV, which requires our careful analysis of limits (see remark~\ref{rem:order-of-limits}).
\end{remark}
Some comments are in order for arguments that are roots of unity. In general we cannot simply substitute such values into the functions \eqref{eq:def-MPLs}, because some factors $\Li_{\vect{n}^{(i)}}(\vect{z}^{(i)})$ might be singular. For example, in \eqref{eq:Li_1,2(z1,z2)} the term $\Li_{1}(z_1) = -\log(1-z_1)$ is not defined at $z_1 = 1$.
Therefore we must approach such arguments as limits from inside the domain \eqref{eq:domain}, where theorems \ref{thm:generalized-parity} and \ref{thm:generalized-parity-MPLs} are valid. It is well-known that divergences only occur when $(n_d,z_d)=(1,1)$, for example see \cite[Lemma~3.3.16]{Panzer:PhD}:
\begin{lemma}
	\label{lem:convergent-limit}%
	Let $\vect{n}\in\N^d$ and $\abs{z_1}=\cdots=\abs{z_{d}}=1$ with $(n_d,z_d)\neq(1,1)$.
	Then the iterated integral \eqref{eq:MPL-as-Hlog} for $\Li_{\vect{n}}(\vect{z})$ converges absolutely and equals the limit
	\begin{equation*}
		\Li_{\vect{n}}(\vect{z})
		=
		\lim_{\vect{y}\rightarrow \vect{z}} \Li_{\vect{n}}(\vect{y}),
	\end{equation*}
	taken from $\vect{y}$ within the domain \eqref{eq:domain} in an arbitrary way.
\end{lemma}
To obtain \eqref{eq:parity-Nth-roots} for given $\vect{n}\in\N^d$ and $\vect{z} \in \RU{N}^d$ with $(n_d,z_d) \neq (1,1)$, we write $\PLi_{\vect{n}}(\vect{y})$ with depth $d-1$ using theorem~\ref{thm:generalized-parity-MPLs} and consider the convergent limits in the order
\begin{equation}
	\PLi_{\vect{n}}(\vect{z})
	= \lim_{y_d\rightarrow z_d} \cdots \lim_{y_1\rightarrow z_1} \PLi_{\vect{n}}(\vect{y}).
	\label{eq:PLi-limit}%
\end{equation}
So in each term $\Li_{\vect{n}^{(1)}}(\vect{y}^{(1)}) \cdots \Li_{\vect{n}^{(s)}}(\vect{y}^{(s)})$ in \eqref{eq:def-MPLs} we compute the limits of the arguments from left to right. Since the arguments are consecutive products, their limiting values are in $\RU{N}$.
By lemma~\ref{lem:convergent-limit} the limits are trivial unless the last argument of one of the factors $\Li_{\vect{n}^{(i)}}(\vect{y}^{(i)})$ approaches unity.
For this case we invoke the standard expansion in logarithms, see \cite{Brown:MZVPeriodsModuliSpaces,Panzer:PhD,IharaKanekoZagier:DerivationDoubleShuffle}.
\begin{theorem}
	\label{thm:anareg}%
	Given $\vect{m} \in \N^s$ and fixed arguments $y_1,\ldots,y_{s-1}$, there exist unique functions $f_0(y_s),\ldots,f_s(y_s)$ that are analytic at $y_s=1$ and fulfil
	\begin{equation}
		\Li_{\vect{m}}(\vect{y})
		= \sum_{k=0}^{s} \log^k(1-y_s) f_k(y_s)
		.
		\label{eq:anareg}%
	\end{equation}
\end{theorem}
The important consequence is that the limit $y_s\rightarrow 1$ is finite if and only if $f_k(1) = 0$ for all $k>0$, and in this case, it equals $\lim_{y_s\rightarrow 1} \Li_{\vect{m}}(\vect{y}) = f_0(1)$.
Since the limits in \eqref{eq:PLi-limit} are convergent, all divergences of the individual terms \eqref{eq:def-MPLs} must cancel each other. For example $\PLi_{1,2}(\vect{y})$ in \eqref{eq:Li_1,2(z1,z2)}, the terms
\begin{equation*}
	\Li_1(y_1) \log^2(-y_1 y_2)
	\quad\text{and}\quad
	-\Li_1(y_1) \log^2(-y_2)
\end{equation*}
with $\Li_1(y_1) = - \log(1-y_1)$ are individually divergent as $y_1\rightarrow 1$ but cancel each other. This cancellation of all divergences implies that we are allowed to just replace each individual $\Li_{\vect{m}}(\vect{y})$ by its regularized limit $f_0(1)$, which can be computed in various ways.
For convenience of the reader we recall here the well-known shuffle regularization \cite{Panzer:PhD,IharaKanekoZagier:DerivationDoubleShuffle}.
The alternative approach via quasi-shuffles \eqref{eq:quasi-shuffle} will be used in section~\ref{sec:depth-3} to deduce \eqref{eq:MZV-depth-3} from \eqref{eq:PLi-depth-3}.
\begin{lemma}
	\label{lem:z_d->1}%
	Let $\vect{m}\in\N^s$ and suppose $y_1,\ldots,y_{s-1} \in \RU{N}$ are given $N$-th roots of unity. Then there exist $x_k \in \MZV^{s-k}_{\abs{\vect{m}}-k}(\RU{N})$ for $0\leq k \leq s$ such that
	\begin{equation}
		\lim_{y_s \rightarrow 1} \left[ 
			\Li_{\vect{m}}(\vect{y})
			- \sum_{k=0}^s x_k \log^k(1-y_s)
		\right]
		= 0
		.
		\label{eq:z_d->1}%
	\end{equation}
\end{lemma}
\begin{proof}
	If $m_s\geq 2$, the limit is convergent by lemma~\ref{lem:convergent-limit} and we can set $x_k=0$ for all $k>0$ and $x_0=\restrict{\Li_{\vect{m}}(\vect{y})}{y_s=1}$. So let us assume $m_s=1$ and consider the iterated integral representation \eqref{eq:MPL-as-Hlog},
	\begin{equation*}
		(-1)^s \Li_{\vect{m}}(\vect{y})
		= \int_0^1 v
		\quad\text{where}\quad
		v = \letter{1/y_s} \letter{0}^{m_{s-1}-1} \letter{1/y_{s-1,s}} \!\cdots \letter{0}^{m_1-1} \letter{1/y_{1,s}}
		.
	\end{equation*}
	Let $1\leq r \leq s$ denote the number of consecutive letters $\letter{1/y_s}$ at the beginning of the word $v$. Equivaletly, $r$ is the largest number such that $m_{i}=1$ for all $i>s-r$ and $y_i = 1$ for all $s>i>s-r$.
	The case $r=s$ corresponds to $x_s=(-1)^s/s!$ and $x_k=0$ for all $k<s$ via the shuffle product \eqref{eq:shuffle-product}, because
	\begin{equation*}
		\int_0^1 \letter{1/y_s}^s
		= \frac{1}{s!} \int_0^1 \left( \letter{1/y_s} \right)^{\shuffle s}
		= \frac{1}{s!} \left( \int_0^1\letter{1/y_s} \right)^s
		= \frac{\log^s (1-y_s)}{s!}
		.
	\end{equation*}
	Otherwise, $r<s$ such that the word $v$ has the form
	\begin{equation*}
		v = \letter{1/y_s}^r \letter{\tau} u
	\end{equation*}
	for some letter $\tau \neq 1/y_s$ and a word $u$. We now apply \eqref{eq:shuffle-reg} with all words reversed and $\sigma_1=\cdots=\sigma_r=1/y_s$. With the multiplicativity of iterated integrals over the shuffle product \eqref{eq:shuffle-product} this gives
	\begin{equation*}
		(-1)^{s} \Li_{\vect{m}}(\vect{y})
		= \sum_{k=0}^r (-1)^{r-k} \frac{\log^{k} (1-y_s)}{k!}
		\int_0^1 \letter{\tau} \left[ u \shuffle \letter{1/y_s}^{r-k} \right]
		.
	\end{equation*}
	After expanding the shuffle product, the iterated integrals multiplying the powers of $\log(1-y_s)$ can be rewritten via \eqref{eq:MPL-as-Hlog} as linear combinations of $\Li_{\vect{\tilde{m}}}(\vect{\tilde{y}})$ where $\vect{\tilde{m}} \in \N^{s-k}$. By lemma~\ref{lem:convergent-limit} these remain finite as $y_s\rightarrow 1$, because either $\tau=0$ such that $\tilde{m}_{s-k} \geq 2$ or $\tilde{y}_{s-k}=1/\tau=y_{s-r} y_s \rightarrow y_{s-r} \neq 1$.
	Knowing \eqref{eq:anareg}, we can therefore set
	\begin{equation*}
		x_k = \frac{(-1)^{s+r-k}}{k!}
		\int_0^1 \left[
			\letter{\tau} \left( u \shuffle \letter{1}^{r-k} \right)
		\right]_{y_s=1}
		\in \MZV^{s-k}_{\abs{\vect{m}}-k}(\RU{N})
		.\qedhere
		\label{eq:shuffle-reg-explicit}%
	\end{equation*}
\end{proof}
\begin{example}
	Consider $\Li_{2,1}(-1,z)=\int_0^1 \letter{1/z}\letter{0}\letter{-1/z}$ as $z\rightarrow 1$. In this case $s=2$, $r=1$ and
	$\letter{1/z}\letter{0}\letter{-1/z} = \letter{1/z} \shuffle \letter{0}\letter{-1/z} - \letter{0} \left( \letter{-1/z} \shuffle \letter{1/z} \right)$
	such that
	\begin{align*}
		\Li_{2,1}(-1,z)
		&= \log(1-z) \int_0^1 \letter{0}\letter{-1/z}
		-\int_0^1 \letter{0} \left( \letter{-1/z}\letter{1/z} + \letter{1/z} \letter{-1/z} \right)
		\\
		&= -\log(1-z) \Li_2(-z)
		- \Li_{1,2}(-1, -z) - \Li_{1,2}(-1, z)
		.
	\end{align*}
	From this we read off the regularized limit $x_0 = -\Li_{1,2}(-1,-1) - \Li_{1,2}(-1,1)$.
\end{example}
This gives an algorithm to evaluate \eqref{eq:PLi-limit} explicitly as an element of $\MZV_{\abs{\vect{n}}}^{d-1}(\RU{N})$: Just substitute the roots of unity $\vect{z}$ into each term of the form \eqref{eq:def-MPLs} and replace each divergent factor $\Li_{\vect{n}^{(i)}}(\vect{z}^{(i)})$ with the corresponding regularization $x_0$ from lemma~\ref{lem:z_d->1}.
\begin{remark}\label{rem:Nth-roots-integer}
	For $\vect{z} \in \RU{N}^d$, the Bernoulli polynomials $\ber{k}{z_{i,d}}$ are also evaluated at $N$-th roots of unity and evaluate to a rational multiple of $ (\imag\pi)^k$.
	Via \eqref{eq:polylog-inversion} they give the values $-2\Realteil \Li_{k}(z_{i,d})$ for even weight $k$ and $-2\imag\Imaginaerteil \Li_{k}(z_{i,d})$ in odd weight.
	With these taking the role of $2\mzv{2k}$ in \eqref{eq:parity-MZV-integer}, theorem~\ref{thm:generalized-parity-MPLs} specializes to an integer coefficient parity theorem for special values of MPL at $N$-th roots of unity.
\end{remark}

\section{Depth two}
\label{sec:depth-2}

The following formula, valid for arbitrary $d,r\in \N$ and $k \in \N_0, \vect{n} \in \N^d$, will be very useful in the sequel to compute iterated integrals $\int \dd z_1/z_1$:
\begin{equation}\label{eq:iterated-primitive}
	\ber{k}{z_{1,d}} \Li_{\vect{n}}(\vect{z})
	=\left(z_1 \partial_{z_1} \right)^r
	\sum_{\mu=0}^k
		\binom{-r}{\mu}
		\ber{k-\mu}{z_{1,d}} \Li_{\vect{n}+(r+\mu)\vect{e}_1}(\vect{z})
	.
\end{equation}
It follows from \eqref{eq:diff-MPL} and $z \partial_{z} \ber{k+1}{z} = \ber{k}{z}$ because $(z_1 \partial_{z_1})^r$ maps the summand to $\sum_{s=\mu}^{k} \binom{r}{s-\mu} \ber{k-s}{z_{1,d}} \Li_{\vect{n}+s\vect{e}_1}(\vect{z})$ and $\sum_{\mu} \binom{-r}{\mu} \binom{r}{s-\mu} = \delta_{s,0}$ where
\begin{equation*}
	\binom{-r}{\mu}
	= \frac{(-r)(-r-1)\cdots(-r-\mu+1)}{\mu!}
	= (-1)^{\mu} \binom{r+\mu-1}{\mu}
	.
\end{equation*}
We can use \eqref{eq:iterated-primitive} to integrate \eqref{eq:PLi_1,n} and obtain an explicit formula for $\PLi_{\vect{n}}(\vect{z})$ for arbitrary values of the indices $\vect{n} \in \N^2$. For a start, \eqref{eq:iterated-primitive} shows that
\begin{equation*}\begin{split}
	G
	&\defas \sum_{\mu=0}^{n_2} (-1)^{\mu} \sum_{\nu=0}^{n_2-\mu} 
		\binom{1-n_1}{\nu}
		\ber{n_2-\mu-\nu}{z_1 z_2}
		\Li_{n_1+\mu+\nu}(z_1)
	- \Li_{n_1}(z_1) \ber{n_2}{z_2}
	\\&\quad
	- \Li_{n_1+n_2}(z_1 z_2)
	+ (-1)^{n_2} \big[
		n_2 \Li_{n_2+1}(1/z_2) \ber{n_1-1}{z_1 z_2}
		+ \Li_{n_2}(1/z_2) \ber{n_1}{z_1 z_2}
	\big]
\end{split}\end{equation*}
fulfils $(z_1 \partial_{z_1})^{n_1 -1} G = \PLi_{1,n_2}(\vect{z}) = (z_1 \partial_{z_1})^{n_1-1} \PLi_{\vect{n}}(\vect{z})$ according to \eqref{eq:diff-PLi}.
Thus, as a function of $z_1$, the difference between $G$ and $\PLi_{\vect{n}}(\vect{z})$ must be a polynomial in $\log(-z_1 z_2)$ and can therefore be recovered from \eqref{eq:reglim} via
$\PLi_{\vect{n}}(\vect{z}) \sim -(-1)^{n_1+n_2} \Li_{\vect{n}}(1/\vect{z})$:
\begin{equation*}
	\PLi_{\vect{n}}(\vect{z})
	\sim
	(-1)^{n_2} \sum_{k=0}^{n_1} \ber{n_1-k}{z_1 z_2} \Li_{n_2+k}(1/z_2) \binom{n_2-1+k}{k}
	.
\end{equation*}
Note that the first summands ($k=0,1$) reproduce the terms of $G$ that are purely logarithmic in $z_1$. Adding the remaining terms to $G$, we arrive at
\begin{align}\label{eq:PLi-depth-2}
	\PLi_{\vect{n}}(\vect{z})
	&=
	\Li_{\vect{n}}(\vect{z})-(-1)^{\abs{\vect{n}}} \Li_{\vect{n}}(1/\vect{z})
	\\
	&=
	(-1)^{n_1}\sum_{\mu=n_1}^{\abs{\vect{n}}}
	\binom{\mu-1}{n_1-1}\Li_{\mu}(z_1) \ber{\abs{\vect{n}}-\mu}{z_1 z_2} (-1)^{\mu}
	- \Li_{\abs{\vect{n}}}(z_1 z_2) 
	\nonumber\\&
	+(-1)^{n_2}\sum_{\mu=n_2}^{\abs{\vect{n}}}
	\binom{\mu-1}{n_2-1} \Li_{\mu}(\tfrac{1}{z_2}) \ber{\abs{\vect{n}}-\mu}{z_1 z_2}
	- \Li_{n_1}(z_1) \ber{n_2}{z_2}
	.\nonumber
\end{align}
In the special case when $\abs{z_1}=\abs{z_2}=1$, this formula was obtained in \cite[Proposition~1.2]{Nakamura:SimpleLerchTornheim}.%
\footnote{%
	Mind the misprint in \cite[equation~(1.3)]{Nakamura:SimpleLerchTornheim}: The term $\zeta(a;-y) \zeta(b,x)$ must read $\zeta(a;x) \zeta(b,-y)$.
}
We will now discuss some specializations to roots of unity which had been discovered before.

In the MZV case ($z_1,z_2\rightarrow1$ and $n_2\geq 2$) with odd weight $w=n_1+n_2$, Bernoulli polynomials with odd index vanish\footnote{Except for $\ber{1}{z_1 z_2}$, but this one cancels between the two sums.} and the even ones become Riemann zeta values \eqref{eq:even-mzv}. Thus \eqref{eq:PLi-depth-2} becomes the well-known \cite{BorweinBorweinGirgensohn:ExplicitEuler,Zagier:232,HuardWilliamsZhang:Tornheim}
\begin{equation}\label{eq:MZV-depth2}
	\mzv{\vect{n}}
	=
	(-1)^{n_1} \sum_{\mathclap{\substack{2s+k = w\\k \geq 3}}}
	\mzv{2s}\mzv{k} \left[ \binom{k-1}{n_1-1} + \binom{k-1}{n_2-1} - \delta_{k,n_1}\right]
	-\frac{\mzv{w}}{2}
	.
\end{equation}
Note that we also get a (less interesting) relation for even weight, namely
\begin{equation}\label{eq:MZV-depth2-even}
	0 =(-1)^{n_1} \sum_{\mathclap{\substack{2s+k=w\\k \geq 2}}}
	\mzv{2s}\mzv{k} \left[ \binom{k-1}{n_1-1} + \binom{k-1}{n_2-1} - \delta_{k,n_1} \right]
	+\frac{\mzv{w}}{2}
	.
\end{equation}
More generally, for alternating sums $z_1,z_2 \rightarrow \pm 1$ with odd weight, we can exploit that $\ber{k}{z} = -\PLi_k(z) = -2\Li_{k}(z) \delta_{k,\text{even}}$ because of $1/z=z$ and \eqref{eq:polylog-inversion} to obtain, for all $\vect{z} \in \set{1,-1}^2$, that (with the convention $\Li_{0}(\pm 1)=-\frac{1}{2}$)
\begin{align}
	\Li_{\vect{n}}(\vect{z})
	&= (-1)^{n_1} \sum_{\mathclap{2s+k=w}}
	\Li_{2s}(z_1 z_2) \left[ 
		\binom{k-1}{n_1-1} \Li_k(z_1) 
		+ \binom{k-1}{n_2-1} \Li_k(z_2) 
	\right]
	\nonumber\\&\quad
	- \tfrac{1}{2} \Li_w (z_1 z_2)
	+ \Li_{n_1}(z_1) \Li_{n_2}(z_2) \delta_{n_2,\text{even}}
	.
	\label{eq:ALT-depth-2}%
\end{align}
This formula had been given in \cite[equation~(75)]{BorweinBradleyBroadhurst:Compendium}. Sometimes the notation $\overline{n_i}$ is used to indicate that $z_i=-1$, then a special case of \eqref{eq:ALT-depth-2} would read
\begin{equation*}\begin{split}
	\mzv{\overline{n_1},n_2}
	&=
	(-1)^{n_1} 
	\sum_{\mathclap{2s+k=w}}
		\mzv{\overline{2s}} \left[
			\binom{k-1}{n_1-1}\mzv{\overline{k}}
			+\binom{k-1}{n_2-1}\mzv{k}
		\right]
	\nonumber\\&\quad
	-\frac{\mzv{\overline{w}}}{2}
	+\mzv{\overline{n_1}}\mzv{n_2} \delta_{n_2,\text{even}}
\end{split}\end{equation*}
where $\mzv{\overline{0}} \defas -1/2$. 
An example of the parity theorem (in depth two) for fourth roots of unity had been worked out in \cite{Lalin:Colored}. That result reads ($n$ odd)
\begin{multline}\label{eq:PLi_n,1-i}
		\PLi_{n,1}(\imag,\imag) + \PLi_{n,1}(\imag,-\imag)
		\\
	=
	-2n\imag\Imaginaerteil \Li_{n+1}(\imag)
	-2\imag\sum_{s=1}^{\frac{n-1}{2}}
	\frac{(\imag\pi)^{2s}(4^s-1)}{(2s)!} B_{2s}
	\Imaginaerteil \Li_{n+1-2s}(\imag)
\end{multline}
and follows as a special case from \eqref{eq:PLi-depth-2}: For the left-hand side we first get
\begin{equation*}
	-2n \Li_{n+1}(\imag) 
	- \mzv{n+1}
	- \mzv{\overline{n+1}}
	- \sum_{s=0}^{\frac{n-1}{2}} \big[
		\Li_{n+1-2s}(\imag) \ber{2s}{1}
		+ \Li_{n+1-2s}(-\imag) \ber{2s}{-1}
	\big]
\end{equation*}
after exploiting $\ber{1}{\imag}+\ber{1}{-\imag} = \log(-\imag)+\log(\imag)=0$ and $\ber{2s+1}{\pm 1}=0$. This expression is purely imaginary since $\PLi_{n,1}(\imag,\pm\imag) = 2\imag\Imaginaerteil \Li_{n,1}(\imag,\pm\imag)$, thus we recover \eqref{eq:PLi_n,1-i} upon projecting on imaginary parts and using
\begin{equation}\label{eq:bernoulli-1/2}
	\frac{\mzv{\overline{2s}}}{\mzv{2s}}
	= \frac{\ber{2s}{-1}}{\ber{2s}{1}}
	= \frac{B_{2s}(1/2)}{B_{2s}}
	= 2^{1-2s}-1.
\end{equation}
The cancellation of the real parts implies a relation between Bernoulli polynomials. More generally, consider $\abs{z_1}=\abs{z_2}=1$ in \eqref{eq:PLi-depth-2} such that $\PLi_{\vect{n}}(\vect{z})$ is imaginary (real) for even (odd) weight $\abs{\vect{n}}$. Taking real (imaginary) parts, we conclude with \eqref{eq:polylog-inversion} and $\ber{k}{1/z} = (-1)^k \ber{k}{z}$ that
\begin{multline*}
	\sum_{\mu=1}^{\abs{\vect{n}}} \ber{\abs{\vect{n}}-\mu}{z_1 z_2} \left[ 
		\binom{\mu-1}{n_1-1} (-1)^{n_1} \ber{\mu}{\overline{z_1}}
		+\binom{\mu-1}{n_2-1} (-1)^{n_2} \ber{\mu}{\overline{z_2}}
	\right]
	\\
	=
	\ber{n_1}{z_1} \ber{n_2}{z_2} + \ber{\abs{\vect{n}}}{z_1 z_2}
	.
\end{multline*}
Via $z_1 = e^{2\pi\imag x}$ and $z_2 = e^{2\pi\imag y}$ this generalizes \eqref{eq:MZV-depth2-even} to the identity ($n_1,n_2\geq 0$)
\begin{multline}\label{eq:bernoulli-identity}
	\sum_{\mu=0}^{\abs{\vect{n}}}
	\binom{\abs{\vect{n}}}{\mu}
	B_{\abs{\vect{n}}-\mu}(x+y)
	\left[ 
		\binom{-n_1}{\mu-n_1}  B_{\mu}(x)
		+\binom{-n_2}{\mu-n_2} B_{\mu}(y)
		-\delta_{\mu,0}
	\right]
	\\
	=
	\binom{\abs{\vect{n}}}{n_1} B_{n_1}(x) B_{n_2}(y)
	.
\end{multline}

\section{Depth three}
\label{sec:depth-3}

Above we pointed out that it is not necessary to keep track of polynomials in $\log(-z_{1,N})$ during the integration process, because these will be recovered from the expansion $z_1\rightarrow 0$ using lemma~\ref{lem:reglim}.
So let us write $f\equiv g$ if $f-g$ is a polynomial in $\log(-z_{1,N})$ and $1/z_1$. According to \eqref{eq:diff-PLi-one},
\begin{equation*}
	\partial_{z_1} \PLi_{1,n_2,n_3} (\vect{z})
	\equiv \frac{\PLi_{n_2,n_3} (z_2, z_3) - \PLi_{n_2,n_3}(z_1 z_2, z_3)}{1-z_1}
	- \frac{\Li_{n_2,n_3}(z_1z_2, z_3)}{z_1}
\end{equation*}
which we can integrate using lemma~\ref{lemma:primitives} after inserting \eqref{eq:PLi-depth-2}. In the result
\begin{equation*}\begin{split}
	\PLi_{1,n_2,n_3}(\vect{z})
	&\equiv
	\Li_1(z_1) \PLi_{n_2,n_3}(z_2,z_3)
	- \Li_{1+n_2,n_3}(z_1 z_2, z_3)
	\\&\quad
	+ \Li_{n_2,1}(z_2,z_1) \ber{n_3}{z_3}
	+ \Li_{n_2+n_3,1}(z_2 z_3, z_1)
	\\&\quad
	-\sum_{\mathclap{\mu+\nu+s= n_2+n_3}}
		\ber{s}{z_1 z_2 z_3}
		\binom{-n_3}{\mu-n_3}
		\Li_{\mu}(1/z_3)
		\Li_{1+\nu}(z_1)
		(-1)^{\mu+\nu} 
	\\&\quad
	-\sum_{\mathclap{\mu+\nu+s=n_2+n_3}}
		\ber{s}{z_1 z_2 z_3}
		\binom{-n_2}{\mu-n_2} 
		\Li_{\mu,1+\nu}(z_2,z_1)
		(-1)^{\nu}
\end{split}\end{equation*}
we use summation indices $\mu,\nu,s\geq 0$ and the quasi-shuffle formula \cite{BorweinBradleyBroadhurstLisonek:SpecialValues,IharaKanekoZagier:DerivationDoubleShuffle}
\begin{equation}\label{eq:quasi-shuffle}
	\Li_{n_2,n_1}(z_2,z_1)
	=
	\Li_{n_1}(z_1) \Li_{n_2}(z_2) - \Li_{n_1,n_2}(z_1,z_2)-\Li_{n_1+n_2}(z_1 z_2) 
\end{equation}
for compactness. Note that even though the terms with non-increasing order among the arguments, like $\Li_{n_2,n_1}(z_2,z_1)$, do not appear as generators in \eqref{eq:def-MPLs}, they nevertheless belong to $\MPLs^2_{n_1+n_2}(z_1,z_2)$ via \eqref{eq:quasi-shuffle}. We use such substitution to present concise formulas.
The subsequent integrations over $\dd z_1/z_1$ according to \eqref{eq:diff-PLi} are immediately solved by \eqref{eq:iterated-primitive} with the result
\begin{align}
	\PLi_{\vect{n}}(\vect{z})
	&\equiv
	\Li_{n_1}(z_1) \PLi_{n_2,n_3}(z_2,z_3)
	- \Li_{n_1+n_2,n_3}(z_1 z_2, z_3)
	\label{eq:PLi-depth3-primitive}\\&\quad
	+ \Li_{n_2,n_1}(z_2,z_1) \ber{n_3}{z_3}
	+ \Li_{n_2+n_3,n_1}(z_2 z_3, z_1)
	\nonumber\\&\quad
	-\sum_{\mathclap{\mu+\nu+s= n_2}}
		\ber{s}{z_1 z_2 z_3}
		\binom{-n_3}{\mu}
		\binom{-n_1}{\nu}
		\Li_{n_3+\mu}(1/z_3)
		\Li_{n_1+\nu}(z_1)
		(-1)^{n_3+\mu}
	\nonumber\\&\quad
	-\sum_{\mathclap{\mu+\nu+s=n_3}}
		\ber{s}{z_1 z_2 z_3}
		\binom{-n_2}{\mu} 
		\binom{-n_1}{\nu} 
		\Li_{n_2+\mu,n_1+\nu}(z_2,z_1)
	.\nonumber
\end{align}
This primitive vanishes at $z_1 \rightarrow 0$, hence its deviation from $\PLi_{\vect{n}}(\vect{z})$ is exactly the expansion of lemma~\ref{lem:reglim}. So by adding
\begin{equation*}
	\PLi_{\vect{n}}(\vect{z})
	\sim
	-
	\sum_{\mathclap{\mu+\nu+s=n_1}} \ber{s}{z_1 z_2 z_3} \binom{-n_2}{\mu}\binom{-n_3}{\nu} \Li_{n_2+\mu,n_3+\nu}(1/\vect{z}') (-1)^{n_2+\mu+n_3+\nu}
\end{equation*}
to \eqref{eq:PLi-depth3-primitive}, we arrive at the final formula
\begin{align}
	\PLi_{\vect{n}}(\vect{z})
	&=
	\Li_{n_1}(z_1) \PLi_{n_2,n_3}(z_2,z_3)
	- \Li_{n_1+n_2,n_3}(z_1 z_2, z_3)
	\label{eq:PLi-depth-3}\\&\quad
	+ \Li_{n_2,n_1}(z_2,z_1) \ber{n_3}{z_3}
	+ \Li_{n_2+n_3,n_1}(z_2 z_3, z_1)
	\nonumber\\&\quad
	-\sum_{\mathclap{\mu+\nu+s= n_2}}
		\ber{s}{z_1 z_2 z_3}
		\binom{-n_3}{\mu}
		\binom{-n_1}{\nu}
		\Li_{n_3+\mu}(1/z_3)
		\Li_{n_1+\nu}(z_1)
		(-1)^{n_3+\mu}
	\nonumber\\&\quad
	-\sum_{\mathclap{\mu+\nu+s=n_3}}
		\ber{s}{z_1 z_2 z_3}
		\binom{-n_2}{\mu} 
		\binom{-n_1}{\nu} 
		\Li_{n_2+\mu,n_1+\nu}(z_2,z_1)
	\nonumber\\&\quad
	-
	\sum_{\mathclap{\mu+\nu+s=n_1}} \ber{s}{z_1 z_2 z_3} \binom{-n_2}{\mu}\binom{-n_3}{\nu} \Li_{n_2+\mu,n_3+\nu}(1/\vect{z}') (-1)^{n_2+\mu+n_3+\nu}
	.\nonumber
\end{align}
Note that we can rewrite $\Li_{\vect{n}}(1/\vect{z}) = (-1)^{\abs{\vect{n}}_0} \left[ \Li_{\vect{n}}(\vect{z}) - \PLi_{\vect{n}}(\vect{z}) \right]$ at any time, so keeping in mind \eqref{eq:quasi-shuffle}, equation \eqref{eq:PLi-depth-3} is the explicit witness of $\PLi_{\vect{n}}(\vect{z}) \in \MPLs^2_{\abs{\vect{n}}}(\vect{z})$.
Inserting \eqref{eq:PLi-depth-2} for $\PLi_{n_2,n_3}(z_2,z_3)$ and taking the limits $z_i \rightarrow 1$ in \eqref{eq:PLi-depth-3} yields the closed formula \eqref{eq:MZV-depth-3} which reduces MZV of even weight and depth three.%
\footnote{%
	The calculation shows that all potential divergences cancel each other as expected for $n_3>1$. The quasi-shuffle \eqref{eq:quasi-shuffle} was used for \eqref{eq:MZV-depth-3} to express the formula in terms of convergent MZV.
}

\begin{remark}
	In analogy to \eqref{eq:MZV-depth2-even} in depth two, we also get a relation from odd weight $\abs{\vect{n}}$ in depth three: $\PLi_{\vect{n}}(1,\ldots,1)=0$ implies an odd-weight depth-two relation given by the right-hand side of \eqref{eq:PLi-depth-3}, which automatically reduces to a depth one relation (theorem~\ref{thm:parity}). Further relations come from the necessary cancellation of imaginary parts (the coefficient of $\ber{1}{z_1 z_2 z_3}$ must cancel when $\vect{z} \rightarrow \vect{1}$), generalizing \eqref{eq:bernoulli-identity}.

	We have not studied all these relations in detail, but it might be interesting to do so in order to find out if these contraints contain any further information about MZV (or other special values of MPL) beyond the parity theorem.
\end{remark}

\bibliography{../bib/qft}

\end{document}